\documentclass{amsart}

 \usepackage[inactive]{srcltx} 
\usepackage{amsmath, amsthm, amscd, amsfonts, amssymb, graphicx, color, extarrows}
\usepackage{tikz-cd}
\usepackage[all]{xy}
\newtheorem{thm}{Theorem}[section]
\newtheorem{cor}[thm]{Corollary}
\newtheorem{prop}[thm]{Proposition}
\theoremstyle{definition}

\newtheorem{exm}[thm]{Example}

\theoremstyle{Lemma}
\newtheorem{lem}[thm]{Lemma}

\newcommand{\U}{\mathcal{U}}

\usepackage{color}

\begin{document}


%





\title[On Lie $n$-centralizers, $n$-commuting linear maps and related mappings ]
{On Lie $n$-centralizers, $n$-commuting linear maps and related mappings of algebras }

\author[Hoger Ghahramani]{ Hoger Ghahramani}
\address{Department of Mathematics, Faculty of Science, University of Kurdistan, P.O. Box 416, Sanandaj, Kurdistan, Iran.}
\email{h.ghahramani@uok.ac.ir; hoger.ghahramani@yahoo.com}

\thanks{{\scriptsize
\hskip -0.4 true cm \emph{MSC(2020)}:16W10; 16W25; 47L10; 47A63; 16R60. 
}}

 \email{}
\email{}

 \address{}

 \email{}

 \thanks{}

 \thanks{}



 \date{}

 \dedicatory{}

 \commby{}


 \begin{abstract}
Let $\U$ be an algebra and for $n\geq 2$, 
\[p_n(a_1, a_2,\cdots , a_n )   =[\cdots [[a_1,a_2], a_3],\cdots,a_n]\] 
be the $(n - 1)$-th commutator of $a_1,\cdots ,a_n\in \U$. The linear mapping $\phi:\U \rightarrow \U$ is said to be a Lie $n$-centralizer if 
\[\phi(p_n(a_1, a_2,\cdots , a_n ))=p_n(\phi(a_1), a_2,\cdots , a_n )\] 
for all $a_1, a_2,\cdots , a_n\in \U$, and the mapping $\varphi:\U \rightarrow \U$ is called an $n$-commuting map if 
\[p_n(\varphi(a), a,\cdots , a)=0\] 
for all $a\in \U$. In this article, we consider several local conditions under which linear mappings on algebras act like Lie $n$-centralizers and we study these linear mappings, Lie $n$-centralizers and $n$-commuting linear maps. With various examples, we compare Lie $n$-centralizers, $n$-commuting linear maps, and linear mappings acting similarly to Lie $n$-centralizers under different local conditions. In the following, we consider the mentioned concepts for linear mappings on standard operator algebras over a complex Banach space $\mathcal{X}$ with $dim\mathcal{X} \geq 2$. We show that the linear mappings on the unital standard operator algebras valid in the considered local conditions are $n$-commuting map, and also, we show that some of the considered local conditions are equivalent for linear mappings on unital standard operator algebras and we determine their structure on these algebras. We describe linear mappings on $B(\mathcal{H})$ where $\mathcal{H}$ is a complex Hilbert space with $dim \mathcal{H}\geq2$ or continuous linear mappings with strong operator topology on unital standard operator algebras under considered local conditions. In addition, Lie $n$-centralizers are characterized on (non-unital or unital) standard operator algebras. As an application of the obtained results, we present some variants of Posner's second theorem on unital standard operator algebras.
\end{abstract}

\maketitle

 \section{\bf Introduction}
The study of mappings related to the Lie structure of associative algebras is one of the topics of interest. Among these mappings, we can mention the Lie centralizer and its generalizations. Let $\mathcal{U}$ be an algebra. The linear map $\phi:\U \rightarrow \U$ is said to be a \textit{Lie centralizer} if $\phi([a,b])=[\phi(a),b]$ for all $a,b\in \U$, where $[a,b]=ab-ba$ is the Lie product of $a$ and $b$. It is easily checked that $\phi$ is a Lie centralizer on $\U$ if and only if $\phi([a,b])=[a,\phi(b)]$ for all $a,b\in \U$. The notion of Lie centralizer is a classical notion in the theory of Lie algebras and other nonassociative algebras (see, \cite{jac, mc}). We are only interested in the case where algebras are associative algebras. Lie centralizers have recently been widely studied from different perspectives on algebras. These studies are aimed at determining the structure of Lie centralizers or characterizing mappings that act at certain products such as Lie centralizers. These study directions are not only interesting on their own, but also appear in connection with other study directions, which we mention a few of them. In \cite{bre}, Bre\v{s}ar has characterized continuous commuting linear maps on a class of Banach algebras in terms of Lie centralizers and then asked about determining the structure of Lie centralizers. In \cite{mag}, in connection with the study of relative commutants of finite groups of unitary operators, Lie centralizers have appeared and the author has characterized the bounded Lie centralizers from a certain type of von Neumann algebras into certain bimodules. The author in \cite{gh hoger} studies commutant or double commutant preserving maps on nest algebras. Indeed, the main results of \cite{gh hoger} come from the characterization of Lie centralizers through two-sided zero products. Lie centralizers also appear in the study of generalized Lie derivations, for which see \cite{beh, goo} and references therein.
\par 
Now, in the following, we will state some of the results obtained for Lie centralizers in line with the mentioned research directions. Fo\v{s}ner and Jing in \cite{fo} have studied non-additive Lie centralizers on triangular rings. Also, in \cite{liu2} non-linear Lie centralizers on generalized matrix algebra have been characterized. In \cite{jab} it has been shown that under some conditions on a unital generalized matrix algebra $\mathcal{G}$, if $\phi:\mathcal{G}\rightarrow \mathcal{G}$ is a linear Lie centralizer, then $\phi(a)=\lambda a+\mu (a) $ in which $\lambda\in \mathrm{Z}(\mathcal{G})$ and $\mu$ is a linear map from $\mathcal{G}$ into $\mathrm{Z}(\mathcal{G})$ vanishing at commutators. Ghimire \cite{ghi} has investigated the linear Lie centralizers of the algebra of dominant block upper triangular matrices. Also, see \cite{bah, ghi0}. In some papers, the following condition on a linear mapping $\phi:\U \rightarrow \U$ is considered
\[ ab=0 \Rightarrow \phi ( [ a,b] ) = [ \phi (a) , b ] = [ a, \phi(b) ]  \quad \mathbf{(C)}.\]
 for $a,b \in \U$. The authors in \cite{beh} have studied the property $\mathbf{(C)}$ for linear mappings on non-unital triangular algebras. In \cite{fad3}, linear mappings satisfying $\mathbf{(C)}$ on a 2-torsion free unital generalized matrix algebra under some mild conditions, are described. The result \cite[Corollary 5.1]{fad3} states that if $\U$ is a unital standard operator algebra on a complex Banach space $\mathcal{X}$, then the linear mapping $ \phi : \mathcal{U} \to \mathcal{U} $ satisfies $\mathbf{(C)}$ if and only if $ \phi (A) = \lambda A + \mu (A)I $ for any $A\in \mathcal{U}$, where $ \lambda \in \mathbb{C} $ (the field of complex numbers) and $\mu: \mathcal{U}\rightarrow \mathbb{C}$ is a linear mapping in which $ \mu ( [ A,B ] ) =0 $, for any $ A,B \in \mathcal{U} $ with $ AB = 0 $. Fo\v{s}ner et al. \cite{fos2} have considered the additive mappings satisfying $\mathbf{(C)}$ on a unital prime ring of characteristic not 2 with a non-trivial idempotent, and obtained similar results. It has also been studied for mappings on some algebras (rings), similar to property $\mathbf{(C)}$ at other products, see \cite{fad2, gh4, goo, liu3}. Ghahramani and Jing \cite{gh-jing} introduced the following concept. A linear mapping $\phi:\U \rightarrow \U$ is said to satisfy condition $\mathbf{(C_1)}$ or $\mathbf{(C_2)}$ if 
 \[ ab=0 \Rightarrow \phi ( [ a,b] ) = [ \phi (a) , b ]   \quad \mathbf{(C_1)};\]
 or 
 \[  ab=0 \Rightarrow \phi ( [ a,b] ) =  [ a, \phi(b) ]  \quad \mathbf{(C_2)}\]
for $a,b \in \U$. The conditions $\mathbf{(C_1)}$ or $\mathbf{(C_2)}$ of linear mappings is weaker than the condition of being a Lie centralizer (see \cite[Example 2.4]{gh-jing}) or satisfying the condition $\mathbf{(C)}$ (see \cite[Example 2.2]{gh-jing}). In \cite{gh-jing}, with various examples, linear mappings that satisfy the conditions $\mathbf{(C_1)}$ or $\mathbf{(C_2)}$, Lie centralizers and commuting linear maps have been compared. A mapping $\varphi : \U\rightarrow \U$ is called \textit{commuting map} if $[\varphi(a),a]=0$ for all $a\in \U$. A commuting map $\varphi : \U\rightarrow \U$ is called \textit{proper} if it is of the form $\varphi(a)=\lambda a+\mu(a)$ for all $a\in\U$, where $\lambda\in \mathrm{Z}(\U)$ and $\mu$ is a linear mapping from $\U$ into $\mathrm{Z}(\U)$. Here $\mathrm{Z}(\U)$ represents the center of $\U$.  The commuting maps have played a crucial role in the development of the theory of functional identities (see \cite{bre0}). The first important result on commuting maps is the Posner’s second theorem \cite{posner} which says that the existence of a nonzero commuting derivation on a prime ring $\mathcal{A}$ implies that A is commutative. A number of authors have extended Posner’s second theorem in several ways. They have showed that nonzero derivations can not be commuting on various subsets on non-commutative prime rings. In \cite{bre 1993} on prime rings commuting additive maps with no further assumptions  have been characterized. We encourage the reader to read the well-written survey paper \cite{bre 2004}, in which the author presented the development of the theory of commuting maps and their applications in detail. It is clear that any Lie centralizer on an algebra is a commuting linear map, but the converse is not necessarily true (see \cite[Example 2.3]{gh-jing}). Characterization of commuting linear maps in terms of Lie centralizers is also one of the issues of interest, as we mentioned, in \cite[Theorem 3.1]{bre} results in this direction are given for a class of Banach algebras. In the continuation of this discussion, it is natural to compare the linear mappings that satisfy the conditions $\mathbf{(C_1)}$ or $\mathbf{(C_2)}$ with the commuting linear maps. It has been shown in \cite{gh-jing} that conditions $\mathbf{(C_1)}$ or $\mathbf{(C_2)}$ are not necessarily equivalent in general (see \cite[Example 2.2]{gh-jing}), and an example has been given (see \cite[Example 2.4]{gh-jing}) to show there exists a linear map satisfying condition $\mathbf{(C_1)}$ and $\mathbf{(C_2)}$ that is not a commuting linear map, and in \cite[Example 2.3]{gh-jing}, they have shown that the commuting linear maps do not necessarily fulfill the conditions $\mathbf{(C_1)}$ and $\mathbf{(C_2)}$. In the following in \cite{gh-jing}, these questions have been raised; what is the relationship between the commuting linear maps and the linear mappings satisfying $\mathbf{(C_1)}$ or $\mathbf{(C_2)}$? On which algebra, each linear map satisfying $\mathbf{(C_1)}$ or $\mathbf{(C_2)}$ becomes a commuting map? On which algebras are conditions $\mathbf{(C_1)}$, and $\mathbf{(C_2)}$ equivalent? Furthermore, the authors introduced a class of operator algebras on Banach spaces such that if $\U$ is in this class, then any linear mapping on $\U$ satisfying $\mathbf{(C_1)}$ or $\mathbf{(C_2)}$ is a commuting map. By this result, they have concluded that conditions $\mathbf{(C_1)}$ and $\mathbf{(C_2)}$ are equivalent for this class of operator algebras, and have characterized linear mapping satisfying $\mathbf{(C_1)}$ or $\mathbf{(C_2)}$ on nest algebras over a complex Hilbert space and on $B(\mathcal{X})$ the algebra of all bounded linear operators on a complex Banach space $\mathcal{X}$. In continuation of the mentioned study path, our intention in this article is to introduce and study the topics of article \cite{gh-jing} for Lie $n$-centralizers and $n$-commuting linear maps.
 \par 
Let $\mathcal{U}$ be an algebra. The \textit{(n - 1)-th commutator} for $n\geq 1$ on $\U$ is denoted by $p_n(a_1, a_2,\cdots , a_n ) $ and is recursively defined as follows:
\[p_1(a_1)=a_1, \quad (a_1\in \U)\]
and for all integers $n\geq 2$
\[p_n(a_1, a_2,\cdots , a_n )   =[ p_{n-1}(a_1, a_2\cdots , a_{n-1}  ), a_n ],\quad (a_1, a_2,\cdots , a_n\in \U). \]
Actually, 
\[p_n(a_1, a_2,\cdots , a_n )   =[\cdots [[a_1,a_2], a_3],\cdots,a_n].\]
It is obvious that $ p_n(a_1, a_2,\cdots , a_n )   = p_{n-1}([a_1 ,a_2],\cdots , a_n  )$ for all $a_1, a_2,\cdots , a_n\in \U$ and $n\geq 2$. The mappings related to the $(n - 1)$-th commutator on the associative algebras are of interest and are used in the study of the structure of these algebras. Among these mappings, we can mention Lie $n$-centralizers. The linear map $\phi:\U \rightarrow \U$ is said to be a \textit{Lie $n$-centralizer} for $n\geq 2$ if 
\[\phi(p_n(a_1, a_2,\cdots , a_n ))=p_n(\phi(a_1), a_2,\cdots , a_n )\] 
for all $a_1, a_2,\cdots , a_n\in \U$. It is easily checked that $\phi$ is a Lie $n$-centralizer on $\U$ if and only if $\phi(p_n(a_1, a_2,\cdots , a_n ))=p_n(a_1,\phi(a_2), \cdots , a_n )$ for all $a_1, a_2,\cdots , a_n\in \U$. In particular, a Lie $2$-centralizer is a Lie centralizer and similarly, a Lie $3$-centralizer is a Lie triple centralizer. It should be noted that a routine check shows that each Lie centralizer is Lie $n$-centralizer for each $n\geq 2$, but the converse is not necessarily true (see \cite[Example 1.2]{fad4}), so Lie $n$-centralizers are non-trivial generalizations of Lie centralizers. Lie $n$-centralizers are not only important as a generalization of Lie centralizers, but also appear in the study of the structure of other mappings such as generalized Lie n-derivations, for instance see \cite{ash}. According to the mentioned cases, recently the study of the structure of Lie $n$-centralizers for different values of n has been taken into consideration. In \cite{fad4} it has been shown that under some conditions on a unital generalized matrix algebra $\mathcal{G}$, if $\phi:\mathcal{G}\rightarrow \mathcal{G}$ is a linear Lie triple centralizer, then $\phi(a)=\lambda a+\mu (a) $ in which $\lambda\in \mathrm{Z}(\mathcal{G})$ and $\mu$ is a linear mapping from $\mathcal{G}$ into $\mathrm{Z}(\mathcal{G})$ vanishing at every second commutator $[[a,b],c]$ for all$a,b,c\in \mathcal{G}$. In \cite{yu}, Lie $n$-centralizers of generalized matrix algebras have been examined. The Lie $n$-centralizers are related to the $n$-commuting maps as well, which we express. A mapping  $\varphi : \U\rightarrow \U$ is called \textit{$n$-commuting map} for $n\geq 2$ if $p_n(\varphi(a), a,\cdots , a)=0$ for all $a\in \U$. It is clear that every commuting map is an $n$-commuting map for each $n\geq 2$ and $n$-commuting maps have been studied as a generalization of commuting maps. Bre\v{s}ar has characterized $3$-commuting additive maps on prime rings of characteristic not 2 in \cite{bre 1992}. Then, in \cite{bei}, the $n$-commuting additive maps for $n\geq2$ are described under certain conditions on the prime rings, and in \cite{liu 2020}, this result is extended to the semiprime rings. We refer to \cite{liu 2020, liu chen} to see more details in this regard. Every Lie $n$-centralizer is a $n$-commuting linear maps, but the converse is not necessarily true (see Example \ref{n-commuting not other}), so the characterization of $n$-commuting linear maps in terms of Lie $n$-centralizers on certain algebras is of interest. In continuation of these studies and with the idea of \cite{gh-jing}, in this article we consider linear mappings that are like Lie $n$-centralizers under some local conditions and compare these mappings with each other as well as with Lie $n$-centralizers and $n$-commuting linear maps. 
\par 
Let $\mathcal{U}$ be an algebra and $\phi:\U \rightarrow \U$ be a linear mapping. The local conditions for $\phi$ that we consider are as follows, where $n\geq 2$ is a fixed positive integer: 
 \begin{equation*}
 \begin{split}
 a_1a_2=0 \Rightarrow\phi(p_n(a_1, a_2,\cdots , a_n ))&=p_n(\phi(a_1), a_2,\cdots , a_n )\\&
 = p_n(a_1,\phi(a_2), \cdots , a_n ) \quad \mathbf{(T)};
 \end{split}
 \end{equation*}
 \[ a_1a_2=0 \Rightarrow \phi(p_n(a_1, a_2,\cdots , a_n ))=p_n(\phi(a_1), a_2,\cdots , a_n )\quad \mathbf{(T_l)};\]
 \[ a_1a_2=0 \Rightarrow \phi(p_n(a_1, a_2,\cdots , a_n ))=p_n(a_1,\phi(a_2), \cdots , a_n )\quad \mathbf{(T_r)}.\]
 for $a_1, a_2,\cdots ,\, a_n\in \U$. Also, we consider the following local conditions on the zero n-th products: 
 \begin{equation*}
 \begin{split}
a_1a_2\cdots a_n=0 \Rightarrow\phi(p_n(a_1, a_2,\cdots , a_n ))&=p_n(\phi(a_1), a_2,\cdots , a_n )\\&
 = p_n(a_1,\phi(a_2), \cdots , a_n ) \quad \mathbf{(N)};
 \end{split}
 \end{equation*} 
 \[a_1a_2\cdots a_n=0 \Rightarrow \phi(p_n(a_1, a_2,\cdots , a_n ))=p_n(\phi(a_1), a_2,\cdots , a_n )\quad \mathbf{(N_l)};\]
 \[a_1a_2\cdots a_n=0 \Rightarrow \phi(p_n(a_1, a_2,\cdots , a_n ))=p_n(a_1,\phi(a_2), \cdots , a_n )\quad \mathbf{(N_r)}.\]
 for $a_1, a_2,\cdots ,\, a_n\in \U$. If $n= 2$, conditions $\mathbf{(N)}$ and $\mathbf{(T)}$ become condition $\mathbf{(C)}$, conditions $\mathbf{(N_l)}$ and $\mathbf{(T_l)}$ become condition $\mathbf{(C_1)}$, and conditions $\mathbf{(N_r)}$ and $\mathbf{(T_r)}$ become condition $\mathbf{(C_2)}$. It is clear that  for $n\geq 2$ every Lie $n$-centralizer satisfies the conditions $\mathbf{(N)}$, $\mathbf{(N_l)}$, $\mathbf{(N_r)}$, $\mathbf{(T)}$, $\mathbf{(T_l)}$ and $\mathbf{(T_r)}$. These results are also obvious: $\mathbf{(N)}\Rightarrow \mathbf{(T)}\Rightarrow \mathbf{(T_l)}$, $\mathbf{(N)}\Rightarrow \mathbf{(T)}\Rightarrow \mathbf{(T_r)}$; $\mathbf{(N)}\Rightarrow \mathbf{(N_l)}\Rightarrow \mathbf{(T_l)}$ and $\mathbf{(N)}\Rightarrow \mathbf{(N_r)}\Rightarrow \mathbf{(T_r)}$. In this article, with various examples (See the Examples Section), we compare the Lie $n$-centralizers, $n$-commuting linear maps, and linear mappings valid in $\mathbf{(N)}$, $\mathbf{(N_l)}$, $\mathbf{(N_r)}$, $\mathbf{(T)}$, $\mathbf{(T_l)}$ and $\mathbf{(T_r)}$. These examples show that in general, the converse of the mentioned requirements are not necessarily true. Also, with these examples, we see that conditions $\mathbf{(N)}$, $\mathbf{(N_l)}$, $\mathbf{(N_r)}$, $\mathbf{(T)}$, $\mathbf{(T_l)}$ and $\mathbf{(T_r)}$ are not necessarily satisfied for an $n$-commuting linear map, and a linear mapping for which one of the conditions $\mathbf{(N)}$, $\mathbf{(N_l)}$, $\mathbf{(N_r)}$, $\mathbf{(T)}$, $\mathbf{(T_l)}$ or $\mathbf{(T_r)}$ is satisfied, is not necessarily an $n$-commuting map (and therefore a $n$-centralizer). Considering these examples, the following questions naturally arise: On what algebras, some conditions $\mathbf{(N)}$, $\mathbf{(N_l)}$, $\mathbf{(N_r)}$, $\mathbf{(T)}$, $\mathbf{(T_l)}$ and $\mathbf{(T_r)}$ are equivalent to each other? On what algebras do some of the conditions $\mathbf{(N)}$, $\mathbf{(N_l)}$, $\mathbf{(N_r)}$, $\mathbf{(T)}$, $\mathbf{(T_l)}$ and $\mathbf{(T_r)}$ result that a linear mapping is $n$-commuting? On what algebras can the structure of the linear mapping valid in one of the conditions $\mathbf{(N)}$, $\mathbf{(N_l)}$, $\mathbf{(N_r)}$, $\mathbf{(T)}$, $\mathbf{(T_l)}$ or $\mathbf{(T_r)}$ be determined? In the rest of the article, we will examine the answers to these questions on standard operator algebras over Banach spaces, and we get the following results. Let $\U$ be a unital standard operator algebra over a complex Banach space $\mathcal{X}$ with $dim \mathcal{X}\geq2$, and $\phi :\U \rightarrow \U$ be a linear mapping. We show that if $\phi $ satisfies one of the conditions $\mathbf{(N)}$, $\mathbf{(N_l)}$, $\mathbf{(N_r)}$, $\mathbf{(T)}$, $\mathbf{(T_l)}$ or $\mathbf{(T_r)}$ for $n\geq 2$, then $\phi$ is an $n$-commuting map (Theorem \ref{conditions n-commuting}). We prove the equivalence of conditions  $\mathbf{(T)}$, $\mathbf{(T_l)}$ and $\mathbf{(T_r)}$ together, and conditions $\mathbf{(N)}$, $\mathbf{(N_l)}$ and $\mathbf{(N_r)}$ together for $\phi$, and more than that, we determine the structure of $\phi$ under these conditions (Theorems \ref{conditions T} and \ref{conditions N}). Our results are also a generalization of the results of \cite{gh-jing}. We determine the structure of Lie $n$-centralizers on $\U$ (Proposition \ref{n-lie unital}). If $\U=B(\mathcal{H})$ (where $\mathcal{H}$ is a complex Hilbert space with $dim \mathcal{H}\geq2$) or $\phi$ is continuous under the strong operator topology, we conclude that conditions $\mathbf{(N)}$, $\mathbf{(N_l)}$, $\mathbf{(N_r)}$, $\mathbf{(T)}$, $\mathbf{(T_l)}$, $\mathbf{(T_r)}$ and that $\phi$ is Lie $n$-centralizer are equivalent (Propositions \ref{hilbert} and \ref{sot}). If $\U$ is non-unital, we characterize $n$-commuting linear maps on $\U$ and the interesting result is that in this case they are equivalent to Lie $n$-centralizers (Theorem \ref{n-lie non-unital}). As an application of the obtained results, we present some variants of Posner's second theorem on unital standard operator algebras (Theorem \ref{posner}).
 \par 
 The article is organized in 5 sections. Section 2 is devoted to the preliminaries and tools. In Section 3, examples are given. In Section 4, conditions $\mathbf{(N)}$, $\mathbf{(N_l)}$, $\mathbf{(N_r)}$, $\mathbf{(T)}$, $\mathbf{(T_l)}$ and $\mathbf{(T_r)}$ for unital standard operator algebras are studied. Lie $n$-centralizers on standard operator algebras are characterized in Section 5.
\section{\bf Preliminaries and tools}
 In this section, we provide some preliminaries and results which are used in other sections.
\par 
Let $\mathcal{R}$ be a ring. By $\mathrm{Z}(\mathcal{R})$ we denote the centre of $\mathcal{R}$. For a semiprime ring $\mathcal{R}$ (i.e., $a\mathcal{R} a = 0$ implies $a = 0$), we denote the maximal right ring of quotients of $\mathcal{R}$ by $\mathcal{Q}_{mr}(\mathcal{R})$. It is known that $\mathcal{Q}_{mr}(\mathcal{R})$ is also a semiprime ring and the centre $\mathcal{C}$ of $\mathcal{Q}_{mr}(\mathcal{R})$, which is called the \textit{extended centroid} of $\mathcal{R}$, is a regular self-injective ring. Moreover, $\mathrm{Z}(\mathcal{R})\subseteq \mathcal{C}$ and $\mathcal{C}$ is a field if $\mathcal{R}$ is a prime ring (i.e., $a\mathcal{R} b = 0$ implies $a = 0$ or $b = 0$). We refer the reader to the book \cite{bei book} for basic terminology and notation. An additive mapping $\delta:\mathcal{R}\rightarrow \mathcal{R}$ is called a \textit{derivation} when $\delta(ab)=a\delta(b)+\delta(a)b$ for all $a,b\in \mathcal{R}$. A mapping $\varphi : \mathcal{R}\rightarrow \mathcal{R}$ is called \textit{centralizing} on a subset $S$ of $\mathcal{R}$ if $[\delta(a),a]\in \mathrm{Z}(\mathcal{R})$ for all $a\in S$. In \cite{posner}, Posner has proved the following theorem, which is known as Posner's second theorem. 
 \begin{thm}$($ \cite[Theorem 2]{posner}$).$\label{pos2}
 Let $\mathcal{R}$ be a prime ring and $\delta:\mathcal{R}\rightarrow \mathcal{R}$ be a centralizing derivation. If $\mathcal{R}$ is non-commutative, then $\delta=0$.
 \end{thm}
 It is clear that every commuting additive map is a centralizing additive map, the converse of this was proved by Bre\v{s}ar on semiprime rings. In fact, we have the following result.
 \begin{lem}$($ \cite[Proposition 3.1]{bre 1993}$).$\label{centralizing}
 Let $\mathcal{R}$ be a 2-torsion free semiprime ring and $\mathcal{J}$ be a Jordan subring of $\mathcal{R}$. If an additive mapping $\varphi :\mathcal{R}\rightarrow \mathcal{R}$ is centralizing on $\mathcal{J}$, then $\varphi$ is commuting on $\mathcal{J}$.
 \end{lem}
In \cite[Theorem 3.2]{bre 1993}, Bre\v{s}ar has characterized the commuting additive maps on prime rings in such a way that if $\mathcal{R}$ is a prime ring and $\varphi:\mathcal{R}\rightarrow \mathcal{R}$ is a commuting additive map, then there exit $\lambda \in \mathcal{C}$ (the extended centroid of $\mathcal{R}$) and an additive mapping $\mu:\mathcal{R}\rightarrow \mathcal{C}$, such that $\varphi(a)=\lambda a +\mu(a)$ for all $a\in \mathcal{R}$. Then, this result found various generalizations, one of the most general of which is the generalization to $n$-commuting additive maps on semiprime rings \cite[Theorem 1.1]{liu 2020} as follows. Let $\mathcal{R}$ be a semiprime ring and $\varphi:\mathcal{R}\rightarrow \mathcal{Q}_{mr}(\mathcal{R})$ be an $n$-commuting additive map, where $n\geq 2$ is a fixed positive integer. Then there exit $\lambda \in \mathcal{C}$ and an additive mapping $\mu:\mathcal{R}\rightarrow \mathcal{C}$, such that $\varphi(a)=\lambda a +\mu(a)$ for all $a\in \mathcal{R}$.
\par 
 An important class of operator algebras which are prime, are standard operator algebras. Throughout this paper all algebras and vector spaces will be over the complex field $\mathbb{C}$. Let $\mathcal{X}$ be a Banach space. We denote by $B(\mathcal{X})$ the algebra of all bounded linear operators on $\mathcal{X}$, and $F(\mathcal{X})$ denotes the ideal of all bounded finite rank operators in $B(\mathcal{X})$. Recall that a \textit{standard operator algebra} is any subalgebra of $B(\mathcal{X})$ which contains $F(\mathcal{X})$. It is clear $B(\mathcal{X})$ is a standard operator algebra. Let $\U$ be a standard operator algebra over the Banach space $\mathcal{X}$. If $dim \mathcal{X}\geq 2$, then $\U$ is non-commutative. An element $E\in \U$ is said to be \textit{idempotent} if $E^2 =E$. A non-trivial idempotent $E$ is an idempotent such that $E\neq 0, I$, where $I$ is the identity operator of $B(\mathcal{X})$. Every standard operator algebra is a prime algebra and contains non-trivial idempotents. Also, we remark that if the standard operator algebra $\U$ contains the identity operator $I$, we have $\mathrm{Z}(\mathcal{U})=\mathbb{C}I$, and if $\U$ is a standard operator algebra over a infinite dimensional complex Banach space $\mathcal{X}$ without the identity operator $I$, then $\mathrm{Z}(\mathcal{U})=\lbrace 0 \rbrace$.
 \par 
 If $\U$ is a unital standard operator algebra over the Banach space $\mathcal{X}$ with $dim \mathcal{X}\geq 2$, then the extended centroid of $\U$ is equal to $\mathbb{C}$. According to this point, \cite[Theorem 1.1]{liu 2020} about standard algebras is as follows, which we have stated in the form of a lemma.
 \begin{lem}\label{n-comuting standard}
 Let $\U$ be a standard operator algebra over a complex Banach space $\mathcal{X}$ with $dim \mathcal{X}\geq2$, and $I\in \U$. Suppose that  $\varphi:\mathcal{U}\rightarrow \mathcal{U}$ is an $n$-commuting linear map, where $n\geq 2$ is a fixed positive integer. Then there exit $\lambda \in \mathbb{C}$ and a linear mapping $\mu:\mathcal{U}\rightarrow \mathbb{C}$, such that $\varphi(A)=\lambda A +\mu(A)$ for all $A\in \mathcal{U}$.
 \end{lem}
 If $\U$ is a non-unital standard operator algebra over an infinite dimensional Banach space $\mathcal{X}$, then the unitization of $\U$ is the algebra $\U^{\sharp}:=\lbrace A+ \lambda I \, : \, A\in \U, \lambda\in \mathbb{C} \rbrace$ which is obtained by adjoining identity to $\U$. The algebra $\U^{\sharp}$ is a unital standard operator algebra. See \cite{chuang, dal} for more on standard operator algebras.
\section{\bf Examples}
In this section, with various examples, we compare the Lie $n$-centralizers, $n$-commuting linear maps, and linear mappings valid in $\mathbf{(N)}$, $\mathbf{(N_l)}$, $\mathbf{(N_r)}$, $\mathbf{(T)}$, $\mathbf{(T_l)}$ and $\mathbf{(T_r)}$. First, in the example below, we can see that the linear mapping that is $n$-commuting for $n\geq 2$ does not necessarily satisfy conditions $\mathbf{(T_l)}$ and $\mathbf{(T_r)}$, and therefore conditions $\mathbf{(N)}$, $\mathbf{(N_l)}$, $\mathbf{(N_r)}$ and $\mathbf{(T)}$ are not necessarily satisfied for it. Also, this mapping is not Lie $n$-centralizer.
\begin{exm}\label{n-commuting not other}
Assume that $\U$ is a unital algebra with unity $1$ and a non-trival idempotent $e$, and there is an element $ a\in \U $ with $ea(1-e)\neq 0$. Consider a linear mapping $\mu: \U\rightarrow \mathbb{C}$ such that $\mu(ea(1-e))\neq 0$. Let $n\geq 2$. Define the linear mapping $\phi:\U \rightarrow \U$ by $\phi(a)=\mu(a)1$. Because $\phi(\U)\subseteq \mathrm{Z}(\U)$, it follows that $\phi$ is an $n$-commuting mapping. We have
\[ p_n(a(1-e), e, 1-e,\cdots , 1-e )=-ea(1-e).\]
Consequently,
\begin{equation*}
\begin{split}
\phi (p_n(a(1-e), e, 1-e,\cdots , 1-e )) & =-\phi ( ea(1-e))\\&
\neq 0 \\&
=p_n(\phi(a(1-e)), e, 1-e,\cdots , 1-e )\\&
=p_n(a(1-e), \phi(e), 1-e,\cdots , 1-e ).
\end{split}
\end{equation*}
According to the above relationship and that $(a(1-e))e=0$, it follows that $\mathbf{(T_l)}$ and $\mathbf{(T_r)}$are not valid for $\phi$.
\end{exm}
We introduce certain types of algebras that we will use in the construction of the following examples. Let $\mathcal{V}$ be an algebra, $\mathcal{M} $ be a left $\mathcal{V}$-module and $\mathcal{N}$ be a right $\mathcal{V}$-module. We consider the following algebras with the usual matrix operations:
\[ \mathfrak{T}_l(\mathcal{V}, \mathcal{M}):=\left\lbrace \begin{bmatrix}
a & x \\
0 & 0
\end{bmatrix} \, : \, a\in \mathcal{V}, x\in \mathcal{M}\right\rbrace\]
and
\[ \mathfrak{T}_r(\mathcal{V}, \mathcal{N}):=\left\lbrace \begin{bmatrix}
0 & y \\
0 & a
\end{bmatrix} \, : \, a\in \mathcal{V}, y\in \mathcal{N}\right\rbrace .\]
For $n\geq 2$ and $1\leq i \leq n$, suppose 
\[C_{i}=\begin{bmatrix}
a_i & x_i \\
0 & 0
\end{bmatrix}\in  \mathfrak{T}_l(\mathcal{V}, \mathcal{M}) \quad  \text{and}  \quad D_{i}=\begin{bmatrix}
0 & y_i \\
0 & a_i
\end{bmatrix}\in  \mathfrak{T}_r(\mathcal{V}, \mathcal{N}).\]
With a routine verification by induction, it can be seen that the following relations hold
\begin{equation}\label{formule left}
\begin{split}
& C_1 C_2\cdots C_n =\begin{bmatrix}
a_1a_2\cdots a_n & a_1a_2\cdots a_{n-1}x_n \\
0 & 0
\end{bmatrix} ;\\&
p_n(C_1, C_2,\cdots ,C_n)=\begin{bmatrix}
p_n(a_1,a_2,\cdots, a_n) & \ast \\
0 & 0
\end{bmatrix};
\end{split}
\end{equation}
and
\begin{equation}\label{formule right}
\begin{split}
& D_1 D_2\cdots D_n =\begin{bmatrix}
0 & y_1a_2\cdots a_{n} \\
0 & a_1a_2\cdots a_n 
\end{bmatrix} ;\\&
p_n(D_1, D_2,\cdots ,D_n)=\begin{bmatrix}
0 & \ast \\
0 & p_n(a_1,a_2,\cdots, a_n)
\end{bmatrix}.
\end{split}
\end{equation}
In the next example, we present a linear mapping that satisfies condition $\mathbf{(N)}$ but is not an $n$-commuting map. Therefore, we see that the validating mapping in conditions $\mathbf{(N)}$, $\mathbf{(N_l)}$, $\mathbf{(N_r)}$, $\mathbf{(T)}$, $\mathbf{(T_l)}$ and $\mathbf{(T_r)}$ is not necessarily an $n$-commuting map.
\begin{exm}\label{other not n-commuting}
We consider the algebra $\U:=\mathfrak{T}_l(\mathbb{C}, \mathbb{C})$. For $n\geq 2$ and $1\leq i \leq n$, let $A_{i}=\begin{bmatrix}
\lambda_i & \gamma_i \\
0 & 0 
\end{bmatrix}\in  \U $. Using \eqref{formule left} and a routine calculation based on induction, it can be seen that
\begin{equation}\label{formule left complex}
\begin{split}
& A_1 A_2\cdots A_n =\begin{bmatrix}
\lambda_1\lambda_2\cdots \lambda_n & \lambda_1\lambda_2\cdots \lambda_{n-1}\gamma_n \\
0 & 0
\end{bmatrix} ;\\&
\text{and} \\&
p_n(A_1, A_2,\cdots ,A_n)=\begin{bmatrix}
0 & (-1)^{n} \lambda_n \lambda_{n-1}\cdots\lambda_3(\lambda_{1}\gamma_2-\lambda_2\gamma_1)\\
0 & 0
\end{bmatrix}.
\end{split}
\end{equation}
Define the linear mapping $\phi : \U\rightarrow \U$ by $\phi(\begin{bmatrix}
\lambda & \gamma \\
0 & 0 
\end{bmatrix}):=\begin{bmatrix}
0 & \lambda \\
0 & 0 
\end{bmatrix}$. Suppose that $A_1 A_2\cdots A_n =0$, where $n\geq 2$ and for $1\leq i \leq n$, $A_{i}=\begin{bmatrix}
\lambda_i & \gamma_i \\
0 & 0 
\end{bmatrix}\in  \U $. It follows from \eqref{formule left complex} that 
\[\phi(p_n(A_1, A_2,\cdots ,A_n))=0 .\]
On the other hand, according to \eqref{formule left complex}, it follows that 
\[p_n(\phi(A_1), A_2,\cdots ,A_n)=\begin{bmatrix}
0 & (-1)^{n+1} \lambda_n \lambda_{n-1}\cdots\lambda_2\lambda_1 \\
0 & 0 
\end{bmatrix}=0\]
and
\[p_n(A_1, \phi(A_2),\cdots ,A_n)=\begin{bmatrix}
0 & (-1)^{n} \lambda_n \lambda_{n-1}\cdots\lambda_1\lambda_2 \\
0 & 0 
\end{bmatrix}=0.\]
So $\phi$ satisfies condition $\mathbf{(N)}$ for each $n\geq 2$. Consider $A=\begin{bmatrix}
1 & 0 \\
0 & 0 
\end{bmatrix}\in  \U $. From the definition of $\phi$ and \eqref{formule left complex}, we get that
\[p_n(\phi(A), A,\cdots ,A)=\begin{bmatrix}
0 & (-1)^{n+1}  \\
0 & 0 
\end{bmatrix}\neq 0.\]
Therefore, for each $n\geq 2$, $\phi$ is not an $n$-commuting map.
\end{exm}
The next example shows that the condition $\mathbf{(T_l)}$ does not necessarily result in the rest of the conditions $\mathbf{(N)}$, $\mathbf{(N_l)}$, $\mathbf{(N_r)}$, $\mathbf{(T)}$ and $\mathbf{(T_r)}$.
\begin{exm}\label{Tl other not }
Let $\U:=\mathfrak{T}_l(\mathbb{C}, \mathbb{C})$. Define the linear mapping $\phi : \U\rightarrow \U$ by $\phi(\begin{bmatrix}
\lambda & \gamma \\
0 & 0 
\end{bmatrix}):=\begin{bmatrix}
\lambda & \lambda \\
0 & 0 
\end{bmatrix}$. For $n\geq 2$ and $1\leq i \leq n$, let $A_{i}=\begin{bmatrix}
\lambda_i & \gamma_i \\
0 & 0 
\end{bmatrix}\in  \U $ with $A_1 A_2=0$. So $\lambda_1 \lambda_2 =0$ and $\lambda_1 \gamma_2=0$. By \eqref{formule left complex} we obtain
\[\phi(p_n(A_1, A_2,\cdots ,A_n))=0 .\]
On the other hand, according to \eqref{formule left complex}, we have
\[p_n(\phi(A_1), A_2,\cdots ,A_n)=\begin{bmatrix}
0 & (-1)^{n+1} \lambda_n \lambda_{n-1}\cdots\lambda_3(\lambda_1 \gamma_2 -\lambda_2 \lambda_1)\\
0 & 0 
\end{bmatrix}=0.\]
Hence $\phi$ satisfies $\mathbf{(T_l)}$. For $n\geq 3$ and $1\leq i \leq n$ when $i\neq 2$, consider $A_{i}:=\begin{bmatrix}
1& 0 \\
0 & 0 
\end{bmatrix}\in  \U $, and $A_2 := \begin{bmatrix}
0 & 1\\
0 & 0 
\end{bmatrix}$. In this case $A_1 A_2 \cdots A_n =0$ and 
\[p_n(\phi(A_1), A_2,\cdots ,A_n)=\begin{bmatrix}
0 & (-1)^{n}  \\
0 & 0 
\end{bmatrix}\neq 0= \phi(p_n(A_1, A_2,\cdots ,A_n)).\]
So $\phi$ does not satisfy $\mathbf{(N_l)}$ for $n\geq 3$. Now, for $n\geq 2$ and $2\leq i \leq n$ we put $A_{i}:=\begin{bmatrix}
1& 0 \\
0 & 0 
\end{bmatrix}\in  \U $, and $A_1 := \begin{bmatrix}
0 & 1\\
0 & 0 
\end{bmatrix}$. Thus $A_1A_2 =0$ and 
\[p_n(A_1,\phi(A_2), A_3,\cdots ,A_n)=\begin{bmatrix}
0 & (-1)^{n+1}  \\
0 & 0 
\end{bmatrix}\neq 0= \phi(p_n(A_1, A_2,\cdots ,A_n)).\]
In this case, condition $\mathbf{(T_r)}$ for $\phi$ is not fulfilled
\end{exm}
In the next example, we present a linear mapping for which $\mathbf{(T_r)}$ is true, but the rest of the conditions $\mathbf{(N)}$, $\mathbf{(N_l)}$, $\mathbf{(N_r)}$, $\mathbf{(T)}$ and $\mathbf{(T_l)}$ are not satisfied.
\begin{exm}\label{Tr other not}
Let $\mathcal{V}$ be the subalgebra of $M_3(\mathbb{C})$ ($3 \times 3$ complex matrices) of the
form 
\[ \mathcal{V}=\left\lbrace \begin{bmatrix}
\lambda & \gamma & \zeta \\
0 & \lambda & \xi \\
0 & 0 & \lambda
\end{bmatrix} \, : \, \lambda, \gamma, \zeta , \xi\in \mathbb{C}\right\rbrace .\]
Consider the algebra $\U:=\mathfrak{T}_l(\mathcal{V}, \mathcal{V})$. Define the linear mapping $\phi : \U\rightarrow \U$ by $\phi(\begin{bmatrix}
R & S \\
0 & 0 
\end{bmatrix}):=\begin{bmatrix}
0 & R \\
0 & 0 
\end{bmatrix}$, ($R,S\in \mathcal{V}$). For $n\geq 3$ and $1\leq i \leq n$, let $A_{i}=\begin{bmatrix}
R_i & S_i \\
0 & 0 
\end{bmatrix}\in  \U $. We have $p_n(R_1, R_2,\cdots ,R_n)=0$, and by \eqref{formule left} 
\begin{equation}\label{formule left 3*3}
\begin{split}
& A_1 A_2\cdots A_n =\begin{bmatrix}
R_1R_2\cdots R_n & R_1 R_2\cdots R_{n-1}S_n \\
0 & 0
\end{bmatrix} ;\\&
\text{and} \\&
p_n(A_1, A_2,\cdots ,A_n)=\begin{bmatrix}
0 & \ast\\
0 & 0
\end{bmatrix}.
\end{split}
\end{equation}
Consequently, $\phi(p_n(A_1, A_2,\cdots ,A_n))=0$. If $A_1 A_2 =0$, then $R_1R_2=0$, $R_1 S_2=0$. In this case, by using \eqref{formule left 3*3} and a routine calculation based on induction, we get 
\[p_n(A_1,\phi(A_2), A_3,\cdots ,A_n)=\begin{bmatrix}
0 & (-1)^{n}R_n R_{n-1}\cdots R_3 R_1 R_2  \\
0 & 0 
\end{bmatrix}=0.\]
Thus $\phi$ satisfies $\mathbf{(T_r)}$. For $n\geq 3$ and $3\leq i \leq n$, put $A_{i}:=\begin{bmatrix}
I & 0 \\
0 & 0 
\end{bmatrix} $ where $I$ is the identity matrix of $M_3(\mathbb{C})$, $A_1 :=\begin{bmatrix}
R_1 & 0 \\
0 & 0 
\end{bmatrix}$, and $A_2 :=\begin{bmatrix}
R_2 & 0 \\
0 & 0 
\end{bmatrix}$ where $R_1:= \begin{bmatrix}
0 & 0 & 0\\
0 & 0 & 1\\
0 & 0& 0
\end{bmatrix}$, and $R_2:= \begin{bmatrix}
0 & 1 & 0\\
0 & 0 & 0\\
0 & 0& 0
\end{bmatrix}$. In this case $A_1 A_2 =0$ because $R_1R_2=0$. By \eqref{formule left 3*3} and the fact that $R_2R_1= \begin{bmatrix}
0 & 0 & 1\\
0 & 0 & 0\\
0 & 0& 0
\end{bmatrix}$, it follows that 
\[p_n(\phi(A_1), A_2,\cdots ,A_n)=\begin{bmatrix}
0 & (-1)^{n+1}R_2R_1  \\
0 & 0 
\end{bmatrix}\neq 0= \phi(p_n(A_1, A_2,\cdots ,A_n)).\]
So $\phi$ does not satisfy $\mathbf{(T_l)}$ for $n\geq 3$. Now, for $n\geq 3$ and $1\leq i \leq n$ when $i\neq 1,3$, we put $A_{i}:=\begin{bmatrix}
I & 0 \\
0 & 0 
\end{bmatrix}\in  \U $, $A_1 :=\begin{bmatrix}
R_1 & 0 \\
0 & 0 
\end{bmatrix}$, and $A_3 :=\begin{bmatrix}
R_3 & 0 \\
0 & 0 
\end{bmatrix}$ where $R_1:= \begin{bmatrix}
0 & 0 & 0\\
0 & 0 & 1\\
0 & 0& 0
\end{bmatrix}$, and $R_3:= \begin{bmatrix}
0 & 1 & 0\\
0 & 0 & 0\\
0 & 0& 0
\end{bmatrix}$. It follows from \eqref{formule left 3*3} that $A_1A_2 \cdots A_n=0$ and 
\[p_n(A_1, \phi(A_2), A_3,\cdots ,A_n)=\begin{bmatrix}
0 & (-1)^{n}R_3R_1  \\
0 & 0 
\end{bmatrix}\neq 0= \phi(p_n(A_1, A_2,\cdots ,A_n)).\]
In this case, condition $\mathbf{(N_r)}$ for $\phi$ is not fulfilled
\end{exm}
In the next example, we provide a linear mapping that satisfies condition $\mathbf{(N_l)}$ but does not satisfy conditions $\mathbf{(N)}$, $\mathbf{(N_r)}$, $\mathbf{(T)}$ and $\mathbf{(T_r)}$.
\begin{exm}\label{Nl other not }
Suppose that $\mathcal{V}$ is the same algebra defined in Example \ref{Tr other not} and we consider algebra $\U$ in the form of $\U:=\mathfrak{T}_r(\mathcal{V}, \mathcal{V})$. For $n\geq 3$ and $1\leq i \leq n$, let $A_{i}=\begin{bmatrix}
0 & S_i \\
0 & R_i 
\end{bmatrix}\in  \U $. We have $p_n(R_1, R_2,\cdots ,R_n)=0$, and by \eqref{formule right} 
\begin{equation}\label{formule right 3*3}
\begin{split}
& A_1 A_2\cdots A_n =\begin{bmatrix}
0 & S_1 R_2\cdots R_{n}\\
0 & R_1R_2\cdots R_n 
\end{bmatrix} ;\\&
\text{and} \\&
p_n(A_1, A_2,\cdots ,A_n)=\begin{bmatrix}
0 & \ast\\
0 & 0
\end{bmatrix}.
\end{split}
\end{equation}
Define the linear mapping $\phi : \U\rightarrow \U$ by $\phi(\begin{bmatrix}
0 & S \\
0 & R 
\end{bmatrix}):=\begin{bmatrix}
0 & R \\
0 & 0 
\end{bmatrix}$, ($R,S\in \mathcal{V}$). It follows from \eqref{formule right 3*3} that $\phi(p_n(A_1, A_2,\cdots ,A_n))=0$. For $n\geq 3$ and for $1\leq i \leq n$, let's assume that $A_{i}=\begin{bmatrix}
0 & S_i \\
0 & R_i 
\end{bmatrix}\in  \U $ is such that $A_1 A_2\cdots A_n =0$. In view of \eqref{formule right 3*3}, we get $S_1 R_2\cdots R_{n}=0$ and $R_1R_2\cdots R_n =0$ and 
\begin{equation*}
\begin{split} 
\phi(p_n(A_1, A_2,\cdots ,A_n))&=0\\&
= \begin{bmatrix}
0 & R_1R_2\cdots R_n   \\
0 & 0 
\end{bmatrix}\\& 
=p_n(\phi(A_1),A_2, A_3,\cdots ,A_n).
\end{split}
\end{equation*}
Therefore $\phi$ satisfies $\mathbf{(N_l)}$. For $n\geq 3$ and $3\leq i \leq n$, put $A_{i}:=\begin{bmatrix}
0 & 0 \\
0 & I 
\end{bmatrix} $, $A_1 :=\begin{bmatrix}
0 & 0 \\
0 & R_1 
\end{bmatrix}$, and $A_2 :=\begin{bmatrix}
0 & 0 \\
0 & R_2 
\end{bmatrix}$ where $R_1:= \begin{bmatrix}
0 & 0 & 0\\
0 & 0 & 1\\
0 & 0& 0
\end{bmatrix}$, and $R_2:= \begin{bmatrix}
0 & 1 & 0\\
0 & 0 & 0\\
0 & 0& 0
\end{bmatrix}$. In this case $A_1 A_2 =0$ because $R_1R_2=0$. So  
\[p_n(A_1,\phi(A_2), A_3\cdots ,A_n)=\begin{bmatrix}
0 & -R_2R_1 \\
0 & 0 
\end{bmatrix}\neq 0= \phi(p_n(A_1, A_2,\cdots ,A_n)).\]
Consequently, $\phi$ does not satisfy $\mathbf{(T_r)}$ for $n\geq 3$.
\end{exm}
The following example shows that condition $\mathbf{(N_r)}$ for a linear mapping does not necessarily result in conditions $\mathbf{(N)}$, $\mathbf{(N_l)}$, $\mathbf{(T)}$ and $\mathbf{(T_l)}$. 
\begin{exm}\label{Nr other not }
We consider the algebra $\U:=\mathfrak{T}_r(\mathbb{C}, \mathbb{C})$. For $n\geq 2$ and $1\leq i \leq n$, let $A_{i}=\begin{bmatrix}
0 & \gamma_i \\
0 & \lambda_i 
\end{bmatrix}\in  \U $. Using \eqref{formule right} and a routine calculation based on induction, we arrive at
\begin{equation}\label{formule right complex}
\begin{split}
& A_1 A_2\cdots A_n =\begin{bmatrix}
0 & \gamma_1\lambda_2\cdots \lambda_n \\
0 & \lambda_1\lambda_2\cdots \lambda_n
\end{bmatrix} ;\\&
\text{and} \\&
p_n(A_1, A_2,\cdots ,A_n)=\begin{bmatrix}
0 & (\gamma_{1}\lambda_2-\gamma_2\lambda_1) \lambda_3 \cdots\lambda_n\\
0 & 0
\end{bmatrix}.
\end{split}
\end{equation}
Define the linear mapping $\phi : \U\rightarrow \U$ by $\phi(\begin{bmatrix}
0 & \gamma \\
0 & \lambda 
\end{bmatrix}):=\begin{bmatrix}
0 & \lambda \\
0 & \lambda
\end{bmatrix}$. For $n\geq 2$ and for $1\leq i \leq n$, let's assume that $A_{i}=\begin{bmatrix}
0 & \gamma_i \\
0 & \lambda_i 
\end{bmatrix}\in  \U $ is such that $A_1 A_2\cdots A_n =0$. According to \eqref{formule right complex}, we have $\gamma_1 \lambda_2 \cdots \lambda_n =0$, $\lambda_1 \lambda_2 \cdots \lambda_n =0$ and 
\begin{equation*}
\begin{split} 
\phi(p_n(A_1, A_2,\cdots ,A_n))&=0\\&
= \begin{bmatrix}
0 & (\gamma_1 \lambda_2 -\lambda_2 \lambda_1)\lambda_3 \cdots \lambda_n  \\
0 & 0 
\end{bmatrix}\\& 
=p_n(A_1,\phi(A_2), A_3,\cdots ,A_n).
\end{split}
\end{equation*}
Therefore $\phi$ satisfies $\mathbf{(N_r)}$. Now, for $n\geq 2$ and $1\leq i \leq n$ with $i\neq 2$ we put $A_{i}:=\begin{bmatrix}
0& 0 \\
0 & 1 
\end{bmatrix}\in  \U $, and $A_2 := \begin{bmatrix}
0 & 1\\
0 & 0 
\end{bmatrix}$. So $A_1A_2 =0$ and 
\[p_n(\phi(A_1),A_2,\cdots ,A_n)=\begin{bmatrix}
0 & -1 \\
0 & 0 
\end{bmatrix}\neq 0= \phi(p_n(A_1, A_2,\cdots ,A_n)).\]
Consequently, $\phi$ does not satisfy $\mathbf{(T_l)}$ for $n\geq 2$.
\end{exm}
The next example shows that for $n\geq 3$, condition $\mathbf{(T)}$ does not necessarily result in conditions $\mathbf{(N_r)}$, $\mathbf{(N_l)}$ and $\mathbf{(N)}$.
\begin{exm}\label{T other not N}
Consider the algebra $\U:=\mathfrak{T}_l(\mathfrak{T}_r(\mathbb{C},\mathbb{C}), \mathfrak{T}_r(\mathbb{C},\mathbb{C}))$. Define the linear mapping $\phi : \U\rightarrow \U$ by $\phi(\begin{bmatrix}
R & S \\
0 & 0 
\end{bmatrix}):=\begin{bmatrix}
0 & \begin{bmatrix}
0 & 0 \\
0 & \lambda 
\end{bmatrix} \\
0 & 0 
\end{bmatrix}$ where $R=\begin{bmatrix}
0 & \gamma \\
0 & \lambda
\end{bmatrix} \in  \mathfrak{T}_r(\mathbb{C},\mathbb{C})$. For $n\geq 3$ and $1\leq i \leq n$, let $A_{i}=\begin{bmatrix}
R_i & S_i \\
0 & 0 
\end{bmatrix}\in  \U $ where $R_i=\begin{bmatrix}
0 & \gamma_i \\
0 & \lambda_i 
\end{bmatrix}\in  \mathfrak{T}_r(\mathbb{C},\mathbb{C})$ and $S_i=\begin{bmatrix}
0 & \gamma^{\prime}_i \\
0 & \lambda^{\prime}_ i
\end{bmatrix}\in  \mathfrak{T}_r(\mathbb{C},\mathbb{C})$. It follows from \eqref{formule left} and \eqref{formule right complex} that 
\[p_n(A_1, A_2,\cdots ,A_n)=\begin{bmatrix}
\begin{bmatrix}
0 & (\gamma_{1}\lambda_2-\gamma_2\lambda_1) \lambda_3 \cdots\lambda_n\\
0 & 0
\end{bmatrix} & \ast \\
0 & 0
\end{bmatrix}.\]
Consequently, $\phi(p_n(A_1, A_2,\cdots ,A_n))=0$. If $A_1 A_2=0$, then $R_1R_2=0$, and hence $\lambda_1 \lambda_2=0$. So
\[ p_n(\phi(A_1), A_2,\cdots ,A_n))=\begin{bmatrix}
0 &(-1)^{n+1}\begin{bmatrix}
0 & \gamma_{n}\lambda_{n-1} \cdots\lambda_2\lambda_1\\
0 & \lambda_{n}\lambda_{n-1} \cdots\lambda_2\lambda_1
\end{bmatrix}  \\
0 & 0
\end{bmatrix}=0 \]
and 
\[ p_n(A_1,\phi(A_2),\cdots ,A_n))=\begin{bmatrix}
0 &(-1)^{n}\begin{bmatrix}
0 & \gamma_{n}\lambda_{n-1} \cdots\lambda_3\lambda_1\lambda_2\\
0 & \lambda_{n}\lambda_{n-1} \cdots\lambda_3\lambda_1\lambda_2
\end{bmatrix}  \\
0 & 0
\end{bmatrix}=0 .\]
Thus $\phi$ satisfies $\mathbf{(T)}$. For $n\geq 3$ and $1\leq i \leq n$, put $A_{i}:=\begin{bmatrix}
R_i & 0 \\
0 & 0 
\end{bmatrix} $ where $R_i:= \begin{bmatrix}
 0 & 0\\
 0 & 1
\end{bmatrix}$ for $1\leq i \leq n-1$, and $R_n:= \begin{bmatrix}
0 & 1 \\
0 & 0 
\end{bmatrix}$. In this case $A_1A_2\cdots A_n=0$,
\[p_n(\phi(A_1),A_2,\cdots ,A_n)=\begin{bmatrix}
0 & (-1)^{n+1}\begin{bmatrix}
0 & 1 \\
0 & 0 
\end{bmatrix} \\
0 & 0 
\end{bmatrix}\neq 0= \phi(p_n(A_1, A_2,\cdots ,A_n))\]
and
\[p_n(A_1,\phi(A_2),\cdots ,A_n)=\begin{bmatrix}
0 & (-1)^{n}\begin{bmatrix}
0 & 1 \\
0 & 0 
\end{bmatrix} \\
0 & 0 
\end{bmatrix}\neq 0= \phi(p_n(A_1, A_2,\cdots ,A_n)).\]
Hence $\phi$ does not satisfy $\mathbf{(N_l)}$ and $\mathbf{(N_r)}$ for $n\geq 3$.
\end{exm}
In the last example, we show that for $n\geq 2$, a Lie $n$-centralizer is not necessarily in standard form.
\begin{exm}\label{non standard form}
Assume that $\mathcal{I}$ is the ideal of $\dfrac{\mathbb{C}[x]}{(x^{3})}$ containing all members in the form $\lambda_1 u+\lambda_2 u^2$, where $(x^{3})$ is the ideal generated by $x^3$ and $u=x+(x^{3})$. Consider the algebra 
\[ \mathcal{U}:=\left\lbrace \begin{bmatrix}
\lambda & q \\
0 & p
\end{bmatrix} \, : \, \lambda\in \mathbb{C}, p\in \mathcal{I},  q \in \dfrac{\mathbb{C}[x]}{(x^{3})}\right\rbrace\]
under the usual matrix operations. It follows from $($\cite[Lemma 1]{kr}$)$ that $\mathrm{Z}(\U)=\lambda I$ where $I=\begin{bmatrix}
1 & 0 \\
0 & 1 
\end{bmatrix}$ is the identity matrix. Define the linear mapping $\phi : \U\rightarrow \U$ by
\[ \phi( \begin{bmatrix}
\lambda & \gamma_1 u+\gamma_2 u^{2} \\
0 & \xi_0 +\xi_1 u+\xi_2 u^{2}
\end{bmatrix})=\begin{bmatrix}
0 & 0 \\
0 & \lambda u^{2}
\end{bmatrix} .\]
It is easily checked that $\phi([R,S])=0=[\phi(R),S]$ for all $R,S\in \U$. So $\phi$ is a Lie centralizer. If $\phi$ is in standard form, the there are central elements $\lambda I$ and $\zeta I$ in $\mathrm{Z}(\U)$ such that 
\[\begin{bmatrix}
0 & 0 \\
0 & u^{2}
\end{bmatrix}=\phi( \begin{bmatrix}
1 & 0 \\
0 & 0
\end{bmatrix})= \lambda\begin{bmatrix}
1 & 0 \\
0 & 0
\end{bmatrix} +\zeta I.  \]
Therefore $u^{2} =\zeta$ and this is a contradiction. Consequently, $\phi$ is not in standard form. Because every Lie centralizer is a Lie $n$-centralizer for every $n\geq 2$, therefore, $\phi$ for $n\geq 2$ is a Lie $n$-centralizer that is not in standard form.
\end{exm}
At the end of this section, we point out that the comparison of Lie $n$-centralizers, $n$-commuting linear maps, and linear mappings valid in $\mathbf{(N)}$, $\mathbf{(N_l)}$, $\mathbf{(N_r)}$, $\mathbf{(T)}$, $\mathbf{(T_l)}$ and $\mathbf{(T_r)}$ for different values of non-negative integers can be one of the issues of interest that we have not addressed here.
\section{\bf Lie $n$-centralizers via local actions on unital standard operator algebras.}
Throughout this section, we assume that $\U$ is a standard operator algebra over a complex Banach space $\mathcal{X}$ with $dim \mathcal{X}\geq2$. In this section, we study Lie $n$-centralizers, $n$-commuting linear maps and validating linear mappings in conditions $\mathbf{(N)}$, $\mathbf{(N_l)}$, $\mathbf{(N_r)}$, $\mathbf{(T)}$, $\mathbf{(T_l)}$ and $\mathbf{(T_r)}$ on standard operator algebras.
\begin{thm}\label{conditions n-commuting}
Let $\U$ contain the identity operator $I$. Assume that the linear mapping $\phi :\U \rightarrow \U$ satisfies one of the conditions $\mathbf{(N)}$, $\mathbf{(N_l)}$, $\mathbf{(N_r)}$, $\mathbf{(T)}$, $\mathbf{(T_l)}$ or $\mathbf{(T_r)}$ for $n\geq 2$. Then $\phi$ is an $n$-commuting map.
\end{thm}
\begin{proof}
Let $\phi$ satisfies $\mathbf{(T_l)}$. Suppose that $E\in \U$ is a rank-one idempotent, $F:=I-E$, and $A_1, A_3,\cdots , A_n \in \U$ are arbitrary elements. Since $(A_1F)E=0$, it follows that 
\[ \phi (p_n(A_1F, E ,A_3,\cdots , A_n ) )  = p_n( \phi (A_1 F ),E,A_3,\cdots , A_n  ) .\]
Hence
\begin{equation}\label{eq1}
\begin{split}
&\phi (p_n(A_1, E ,A_3,\cdots , A_n ) )-p_n( \phi (A_1 ),E, A_3,\cdots , A_n  )\\&
\, =\phi (p_n(A_1E, E ,A_3,\cdots , A_n ) )+\phi (p_n(A_1F, E ,A_3,\cdots , A_n ) )\\&
\quad -p_n(\phi(A_1E), E ,A_3,\cdots , A_n )-p_n( \phi (A_1 F ),E ,A_3,\cdots , A_n  ) \\ &
\, =\phi (p_n(A_1E, E ,A_3,\cdots , A_n ) )-p_n(\phi(A_1E), E ,A_3,\cdots , A_n ).
\end{split}
\end{equation}
By the fact that $(A_1 E)F=0$ we get
\[ \phi (p_n(A_1E, F ,A_3,\cdots , A_n ) )  = p_n( \phi (A_1 E ),F,A_3,\cdots , A_n  ) .\]
It follows from \eqref{eq1} and above equation that 
\begin{equation*}
\begin{split} 
&\phi (p_n(A_1, E ,A_3,\cdots , A_n ) )-p_n( \phi (A_1 ),E ,A_3,\cdots , A_n  )\\&
\, = \phi (p_n(A_1E, E ,A_3,\cdots , A_n ) )-p_n(\phi(A_1E), E ,A_3,\cdots , A_n ) \\&
\quad + \phi (p_n(A_1E, F ,A_3,\cdots , A_n ) )  - p_n( \phi (A_1 E ),F,A_3,\cdots , A_n  ) \\&
\, =\phi (p_n(A_1E, I ,A_3,\cdots , A_n ) )  - p_n( \phi (A_1 E ),I,A_3,\cdots , A_n  ) \\&
\, =0.
\end{split}
\end{equation*}
Since every finite-rank operator on $\mathcal{X}$ is a linear combination of rank-one idempotents (see \cite[Lemma 1.1]{bur}), we arrive at 
\[ \phi (p_n(A_1, T ,A_3,\cdots , A_n ) )  = p_n( \phi (A_1  ),T,A_3,\cdots , A_n  ) \]
for all $A_1, A_3,\cdots , A_n \in \U$ and $T\in F(\mathcal{X})$.  By the fact that $p_n(A_1, A_2,\cdots , A_n )   = p_{n-1}([A_1 ,A_2],\cdots , A_n  ) $ we have
\begin{equation}\label{eq2}
\begin{split}
\phi (p_{n-1}(A_1 T ,A_3,\cdots , A_n ) ) & =\phi (p_{n-1}(TA_1  ,A_3,\cdots , A_n ) )\\&
\quad +p_{n-1}( \phi (A_1) T,A_3,\cdots , A_n  )\\&
\quad -p_{n-1}(T \phi (A_1) ,A_3,\cdots , A_n  )
\end{split}
\end{equation}
for all $A_1, A_3,\cdots , A_n \in \U$ and $T\in F(\mathcal{X})$. 
\par 
Let $A_1, A_2,\cdots , A_n \in \U$ and $T\in F(\mathcal{X})$. By applying \eqref{eq2}, we have
\begin{equation*}
\begin{split}
\phi (p_{n-1}(A_1 A_2 T ,A_3,\cdots , A_n ) ) & =\phi (p_{n-1}(TA_1 A_2 ,A_3,\cdots , A_n ) )\\&
\quad+p_{n-1}( \phi (A_1 A_2)T,A_3,\cdots , A_n  )\\&
\quad -p_{n-1}( T\phi (A_1 A_2) ,A_3,\cdots , A_n  )
\end{split}
\end{equation*}
and on the other hand 
\begin{equation*}
\begin{split}
\phi (p_{n-1}(A_1 A_2 T ,A_3,\cdots , A_n ) ) & =\phi (p_{n-1}(A_2TA_1  ,A_3,\cdots , A_n ) )\\&
\quad+p_{n-1}( \phi (A_1 )A_2T,A_3,\cdots , A_n  )\\&
\quad -p_{n-1}( A_2T\phi (A_1 ) ,A_3,\cdots , A_n  )\\&
=\phi (p_{n-1}(TA_1 A_2 ,A_3,\cdots , A_n ) )\\&
\quad +p_{n-1}( \phi (A_2 )TA_1,A_3,\cdots , A_n  )\\&
\quad- p_{n-1}( TA_1\phi (A_2) ,A_3,\cdots , A_n  )\\&
\quad+p_{n-1}( \phi (A_1 )A_2T,A_3,\cdots , A_n  )\\&
\quad -p_{n-1}( A_2T\phi (A_1 ) ,A_3,\cdots , A_n  )
\end{split}
\end{equation*}
By comparing the two expressions for $\phi (p_{n-1}(A_1 A_2 T ,A_3,\cdots , A_n ) )$, we see that
\begin{equation*}
\begin{split}
&p_{n-1}( \phi (A_1 A_2)T,A_3,\cdots , A_n  )-p_{n-1}( T\phi (A_1 A_2) ,A_3,\cdots , A_n  )\\&
\,\, -p_{n-1}( \phi (A_2 )TA_1,A_3,\cdots , A_n  )+ p_{n-1}( TA_1\phi (A_2) ,A_3,\cdots , A_n  )\\&
\,\,-p_{n-1}( \phi (A_1 )A_2T,A_3,\cdots , A_n  ) +p_{n-1}( A_2T\phi (A_1 ) ,A_3,\cdots , A_n  )=0
\end{split}
\end{equation*}
Since $\overline{F(\mathcal{X})}^{SOT}=B(\mathcal{X})$, there is a net $(T_{\alpha})_{\alpha\in\Gamma}$ in $F(\mathcal{X})$ converges to $I$ with respect to the strong operator topology. By the fact that the product in $B(\mathcal{X})$ is separately continuous with respect to the strong operator topology we get
\begin{equation*}
\begin{split}
& -p_{n-1}( \phi (A_2 )A_1,A_3,\cdots , A_n  )+ p_{n-1}( A_1\phi (A_2) ,A_3,\cdots , A_n  )\\&
\,\,-p_{n-1}( \phi (A_1 )A_2,A_3,\cdots , A_n  ) +p_{n-1}( A_2\phi (A_1 ) ,A_3,\cdots , A_n  )=0
\end{split}
\end{equation*}
By letting $A:=A_1=A_2=\cdots = A_n$ we have  
\[p_{n-1}( [A,\phi (A)] ,A,\cdots , A  )=0.\]
So
\[ p_{n}( \phi (A) ,A,\cdots , A  )=0\]
for all $A\in \U$, ans $\phi$ is an $n$-commuting map.
\par 
Now we assume that $\phi$ satisfies $\mathbf{(T_r)}$. For $A_2, A_3,\cdots , A_n \in \U$ and a rank-one idempotent $E\in \U$ we have $E(FA_2)=0$ and $F(EA_2)=0$, where  $F:=I-E$. By the hypothesis we see that
\[ \phi (p_n(E, FA_2 ,A_3,\cdots , A_n ) )  = p_n( E,\phi (FA_2 ),A_3,\cdots , A_n  ) \]
and
\[ \phi (p_n(F, EA_2 ,A_3,\cdots , A_n ) )  = p_n( F,\phi (EA_2 ),A_3,\cdots , A_n  ). \]
Hence
\begin{equation*}
\begin{split}
&\phi (p_n(E, A_2 ,A_3,\cdots , A_n ) )  - p_n( E,\phi (A_2 ),A_3,\cdots , A_n  ) \\&
\, =\phi (p_n(E, EA_2 ,A_3,\cdots , A_n ) )  +\phi (p_n(E, FA_2 ,A_3,\cdots , A_n ) ) \\&
\quad -p_n( E,\phi (EA_2 ),A_3,\cdots , A_n  ) -p_n( E,\phi (FA_2 ),A_3,\cdots , A_n  )\\&
\,= \phi (p_n(E, EA_2 ,A_3,\cdots , A_n ) ) -p_n( E,\phi (EA_2 ),A_3,\cdots , A_n  )\\&
\quad +\phi (p_n(F, EA_2 ,A_3,\cdots , A_n ) )  - p_n( F,\phi (EA_2 ),A_3,\cdots , A_n  )\\&
\,= \phi (p_n(I, EA_2 ,A_3,\cdots , A_n ) )  - p_n( I,\phi (EA_2 ),A_3,\cdots , A_n  )\\&
\,=0.
\end{split}
\end{equation*}
So
\[\phi (p_n(T, A_2 ,A_3,\cdots , A_n ) )  = p_n( T,\phi (A_2 ),A_3,\cdots , A_n  ) \]
for all $A_1, A_3,\cdots , A_n \in \U$ and $T\in F(\mathcal{X})$. 
 Thus 
\begin{equation}\label{eq3}
\begin{split}
\phi (p_{n-1}(A_2 T ,A_3,\cdots , A_n ) ) & =\phi (p_{n-1}(TA_2  ,A_3,\cdots , A_n ) )\\&
\quad +p_{n-1}( \phi (A_2) T,A_3,\cdots , A_n  )\\&
\quad -p_{n-1}(T \phi (A_2) ,A_3,\cdots , A_n  )
\end{split}
\end{equation}
for all $A_2, A_3,\cdots , A_n \in \U$ and $T\in F(\mathcal{X})$.  Equation~\eqref{eq3} is the same as Equation~\eqref{eq2}. So with a process similar to the proof of the previous case, we get the conclusion. 
\par 
If $\phi$ achieves one of the conditions $\mathbf{(N)}$, $\mathbf{(N_l)}$, $\mathbf{(N_r)}$ or $\mathbf{(T)}$, then $\mathbf{(T_l)}$ or $\mathbf{(T_r)}$ is valid for it, and therefore, under these conditions, $\phi$ is an $n$-commuting map. Proof of the theorem is complete.
\end{proof}
According to Example \ref{n-commuting not other}, the converse of the above theorem is not necessarily true. In the next theorem, we see that conditions $\mathbf{(T)}$, $\mathbf{(T_l)}$ and $\mathbf{(T_r)}$ for a linear mapping on a standard operator algebra are equivalent, and we also characterize the valid mapping in these conditions on a standard operator algebra.
\begin{thm}\label{conditions T}
Let $\U$ contain the identity operator $I$. Suppose that $\phi :\U \rightarrow \U$ is a linear map. The following are equivalent for $n\geq 2$:
\begin{itemize}
\item[(i)] $\phi$ satisfies $\mathbf{(T)}$;
\item[(ii)] $\phi$ satisfies $\mathbf{(T_l)}$;
\item[(iii)] $\phi$ satisfies $\mathbf{(T_r)}$;
\item[(iv)] There exist an element $\lambda \in \mathbb{C}$ and a linear map $\mu :\U \rightarrow \mathbb{C}$ such that $\phi(A)=\lambda A +\mu (A)I$ for all $A\in \U$ and $\mu (p_n(A_1, A_2,\cdots , A_n ))=0$ for all $A_1, A_2,\cdots , A_n \in \U$ with $A_1 A_2=0$.
\end{itemize}
\end{thm}
\begin{proof}
(i)$\Rightarrow$(ii), (i)$\Rightarrow$(iii), (iv)$\Rightarrow$(i), (iv)$\Rightarrow$(ii) and (iv)$\Rightarrow$(iii) are clear.
\par 
(ii)$\Rightarrow$(iv): By Theorem~\ref{conditions $n$-commuting}, $\phi$ is an n-commuting linear map. From Lemma \ref{n-comuting standard} it follows that $\phi(A)=\lambda A +\mu(A)I$ ($A\in \mathcal{U}$), where $\lambda\in \mathbb{C}$  and $\mu$ is a linear map from $\U$ into $\mathbb{C}$. Now for any $A_1, A_2,\cdots , A_n \in \U$ with $A_1 A_2=0$. we have
\begin{equation*}
\begin{split}
\mu (p_n(A_1, A_2,\cdots , A_n ))&=\phi(p_n(A_1, A_2,\cdots , A_n ))-\lambda p_n(A_1, A_2,\cdots , A_n )\\&
=p_n(\phi (A_1 ),A_2,\cdots , A_n  )-\lambda p_n(A_1, A_2,\cdots , A_n ) \\&
=p_n(\mu ( A_1 ),A_2,\cdots , A_n  )+p_n(\lambda A_1, A_2,\cdots , A_n )\\&
\quad -\lambda p_n(A_1, A_2,\cdots , A_n )  =0.
\end{split}
\end{equation*}
(iii)$\Rightarrow$(iv): Similar to (ii)$\Rightarrow$(iv), the proof is obtained. 
\end{proof}
The following theorem is about conditions $\mathbf{(N)}$, $\mathbf{(N_l)}$ and $\mathbf{(N_r)}$ on standard operator algebras, which is proved in the same way as Theorem \ref{conditions T}.
\begin{thm}\label{conditions N}
Let $\U$ contain the identity operator $I$. Suppose that $\phi :\U \rightarrow \U$ is a linear map. The following are equivalent for $n\geq 2$:
\begin{itemize}
\item[(i)] $\phi$ satisfies $\mathbf{(N)}$;
\item[(ii)] $\phi$ satisfies $\mathbf{(N_l)}$;
\item[(iii)] $\phi$ satisfies $\mathbf{(N_r)}$;
\item[(iv)] There exist an element $\lambda \in \mathbb{C}$ and a linear map $\mu :\U \rightarrow \mathbb{C}$ such that $\phi(A)=\lambda A +\mu (A)I$ for all $A\in \U$ and $\mu (p_n(A_1, A_2,\cdots , A_n ))=0$ for all $A_1, A_2,\cdots , A_n \in \U$ with $A_1 A_2\cdots  A_n =0$.
\end{itemize}
\end{thm}
We don't know if conditions $\mathbf{(N)}$, $\mathbf{(N_l)}$, $\mathbf{(N_r)}$, $\mathbf{(T)}$, $\mathbf{(T_l)}$ and $\mathbf{(T_r)}$ for linear mappings on unital standard operator algebras are equivalent or not for $n\geq 3$? So this case remains as a question. Also, if the answer to this question is negative, the next question that seems to be that on which unital standard operator algebras are the mentioned conditions for the linear mappings equivalent? In the last section, we will answer the second question on specific unital standard operator algebras.
\par 
As we said in the introduction, examples are presented in \cite{gh-jing} that show that conditions $\mathbf{(C)}$, $\mathbf{(C_1)}$, and $\mathbf{(C_2)}$ are not necessarily equivalent or do not necessarily result that a linear mapping is commuting. In the following, these questions have been raised; on which algebras are conditions $\mathbf{(C)}$, $\mathbf{(C_1)}$, and $\mathbf{(C_2)}$ equivalent, and on which algebras do these conditions result in a mapping being commuting? The results of this section for $n=2$ are actually a positive answer to these questions on standard operator algebras which are expressed as a result below. It should also be noted that this result is a generalization of \cite[Corollary 5.1]{fad3}.
\begin{cor}\label{conditions C}
Let $\U$ contain the identity operator $I$. Assume that $\phi :\U \rightarrow \U$ is a linear mapping. Then
\begin{itemize}
\item[(a)] if $\phi $ satisfies one of the conditions $\mathbf{(C)}$, $\mathbf{(C_1)}$ or $\mathbf{(C_2)}$ for $n\geq 2$, then $\phi$ is a commuting map;
\item[(b)] the following are equivalent:
\begin{itemize}
\item[(i)] $\phi$ satisfies $\mathbf{(C)}$;
\item[(ii)] $\phi$ satisfies $\mathbf{(C_1)}$;
\item[(iii)] $\phi$ satisfies $\mathbf{(C_2)}$;
\item[(iv)] There exist an element $\lambda \in \mathbb{C}$ and a linear map $\mu :\U \rightarrow \mathbb{C}$ such that $\phi(A)=\lambda A +\mu (A)I$ for all $A\in \U$ and $\mu ([A_1, A_2])=0$ for all $A_1, A_2  \in \U$ with $A_1 A_2 =0$.
\end{itemize}
\end{itemize}
\end{cor}
At the end of this section, as an application of the obtained results, we present some variants of Posner's second theorem on standard operator algebras.
\begin{thm}\label{posner}
Let $\U$ contain the identity operator $I$. Suppose that $\delta :\U \rightarrow \U$ is a linear derivation. The following are equivalent for $n\geq 2$:
\begin{itemize}
\item[(i)] $\delta$ is a Lie $n$-centralizer
\item[(ii)] $\delta$ satisfies $\mathbf{(T)}$;
\item[(iii)] $\delta$ satisfies $\mathbf{(T_l)}$;
\item[(iv)] $\delta$ satisfies $\mathbf{(T_r)}$;
\item[(v)] $\delta$ satisfies $\mathbf{(N)}$;
\item[(vi)] $\delta$ satisfies $\mathbf{(N_l)}$;
\item[(vii)] $\delta$ satisfies $\mathbf{(N_r)}$;
\item[(viii)] $\delta$ is an $n$-commuting map;
\item[(ix)] $\delta$ is a commuting map;
\item[(x)] $\delta$ is a centralizing map;
\item[(xi)] $\delta =0$.
\end{itemize}
\end{thm}
\begin{proof}
According to Theorem \ref{conditions n-commuting}, conditions (i) to (vii) result in condition (viii). It clearly follows from Lemma \ref{n-comuting standard} that an $n$-commuting linear map is a commuting linear map. So (viii)$\Rightarrow$(ix) is obtained. Since $dim \mathcal{X}\geq2$, then the algebra $\U$ is non-commutative and (x)$\Rightarrow$(xi) is clear according to Theorem \ref{pos2}. (ix)$\Leftrightarrow$(x) is obtained from Lemma \ref{centralizing}. Finally, it is obvious that if $\delta = 0$, then all cases (i) to (ix) are true. In this way, the proof of the theorem is complete.
\end{proof}
The question that naturally arises is whether the same as Theorem \ref{posner} holds for additive derivations on a prime ring with suitable torsion restrictions?
\section{\bf Lie $n$-centralizers on standard operator algebras}
In this section, we will study the $n$-commuting linear mappings and Lie $n$-centralizers on unital and non-unital standard operator algebras. 
\begin{prop}\label{n-lie unital}
Let $\U$ contain the identity operator $I$. Suppose that $\phi :\U \rightarrow \U$ is a linear map. The following are equivalent for $n\geq 2$:
\begin{itemize}
\item[(i)] $\phi$ is a Lie $n$-centralizer;
\item[(ii)] There exist an element $\lambda \in \mathbb{C}$ and a linear map $\mu :\U \rightarrow \mathbb{C}$ such that $\phi(A)=\lambda A +\mu (A)I$ for all $A\in \U$ and $\mu (p_n(A_1, A_2,\cdots , A_n ))=0$ for all $A_1, A_2,\cdots , A_n \in \U$.
\end{itemize}
\end{prop}
\begin{proof}
(i)$\Rightarrow$(ii): Since $\phi$ is a Lie $n$-centralizer, it follows that $\phi$ satisfies $\mathbf{(T)}$. So by Theorem \ref{conditions T} there is an element $\lambda \in \mathbb{C}$ and a linear map $\mu :\U \rightarrow \mathbb{C}$ such that $\phi(A)=\lambda A +\mu (A)$ for all $A\in \U$. Finally, similar to the proof of Theorem \ref{conditions T}, it is shown that $\mu (p_n(A_1, A_2,\cdots , A_n ))=0$ for all $A_1, A_2,\cdots , A_n \in \U$.
\par 
(ii)$\Rightarrow$(i): is clear.
\end{proof}
The above result has been obtained for $n=2$ on the unital standard operator algebras in \cite[Corollary 5.2]{fad3}. Therefore, this result is a generalization of \cite[Corollary 5.1]{fad3}.
\par 
As mentioned in Section 3, we do not know whether conditions $\mathbf{(N)}$, $\mathbf{(N_l)}$, $\mathbf{(N_r)}$, $\mathbf{(T)}$, $\mathbf{(T_l)}$ and $\mathbf{(T_r)}$ are equivalent for linear mappings on unital standard operator algebras or not? In the following, in certain cases, we show that these conditions are equivalent for linear mappings, and we show even more than this. In fact, in these special cases, we show that these conditions are equivalent for linear mappings with the condition that the mappings are Lie $n$-centralizers, and we also characterize linear mappings under these conditions. The following lemma plays an essential role in proving the next two results.
\begin{lem}\label{central map}
Let $\U$ be a unital standard operator algebra over a complex Banach space $\mathcal{X}$ with $dim \mathcal{X}\geq2$. Suppose that $\mu: \U\rightarrow \mathbb{C}$ be a linear map such that $\mu(p_n(A_1,A_2, \cdots A_n))=0$ for $n\geq 2$, $A_1, A_2, \cdots ,A_n\in \U$ with $A_1 A_2 =0$. Then 
\[\mu(p_n(A_1,E,A_3, \cdots A_n))= \mu(p_n(E,A_2,A_3 \cdots A_n))=0\]
for all idempotents $E\in \U$ and all $A_1, A_2, \cdots ,A_n\in \U$.
\end{lem}
\begin{proof}
Suppose that $E\in \U$ is an idempotent, $F:=I-E$, and $A_1, A_2,\cdots , A_n \in \U$ are arbitrary elements. Since $(A_1F)E=(A_1 E)F=0$, it follows that 
\begin{equation*}
\begin{split}
\mu(p_n(A_1,E,A_3, \cdots A_n))&=\mu(p_n(A_1E,E,A_3, \cdots A_n))+\mu(p_n(A_1F,E,A_3, \cdots A_n))\\&
=\mu(p_n(A_1E,E,A_3, \cdots A_n))+\mu(p_n(A_1E,F,A_3, \cdots A_n))\\&
=\mu(p_n(A_1E,I,A_3, \cdots A_n))\\&
=0.
\end{split}
\end{equation*}
Considering that $E(FA_2)=F(EA_2)=0$ and using the same method as before, it is proved that
\[ \mu(p_n(E,A_2,A_3 \cdots A_n))=0.\]
\end{proof}
In the next proposition, we show that for a linear mapping on the algebra of bounded operators on a complex Hilbert space, the conditions $\mathbf{(N)}$, $\mathbf{(N_l)}$, $\mathbf{(N_r)}$, $\mathbf{(T)}$, $\mathbf{(T_l)}$, $\mathbf{(T_r)}$ and that the mapping is Lie $n$-centralizer are equivalent.
\begin{prop}\label{hilbert}
Let $\mathcal{H}$ be a complex Hilbert space with $dim \mathcal{H}\geq2$. Suppose that $\phi :B(\mathcal{H}) \rightarrow B(\mathcal{H})$ is a linear mapping. The following are equivalent for $n\geq 2$:
\begin{itemize}
\item[(i)] $\phi$ is a Lie $n$-centralizer;
\item[(ii)] $\phi$ satisfies $\mathbf{(N)}$;
\item[(iii)] $\phi$ satisfies $\mathbf{(N_l)}$;
\item[(iv)] $\phi$ satisfies $\mathbf{(N_r)}$;
\item[(v)] $\phi$ satisfies $\mathbf{(T)}$;
\item[(vi)] $\phi$ satisfies $\mathbf{(T_l)}$;
\item[(vii)] $\phi$ satisfies $\mathbf{(T_r)}$;
\item[(viii)] There exist an element $\lambda \in \mathbb{C}$ and a linear map $\mu :B(\mathcal{H})\rightarrow \mathbb{C}$ such that $\phi(A)=\lambda A +\mu (A)I$ for all $A\in B(\mathcal{H})$ and $\mu (p_n(A_1, A_2,\cdots , A_n ))=0$ for all $A_1, A_2,\cdots , A_n \in B(\mathcal{H})$.
\end{itemize}
\end{prop}
\begin{proof}
It is clear that (i) results in the rest of the cases. According to Theorem \ref{conditions N}, (ii), (iii), and (iv) are equivalent, and according to Theorem \ref{conditions T}, (v), (vi) and (vii) are equivalent. (i)$\Leftrightarrow$(vii) is obtained from Corollary \ref{n-lie unital}. (ii)$\Rightarrow$(v) is clear. So we just need to prove (vi)$\Rightarrow$(viii) to get the result.
\par 
Suppose that $\phi$ satisfies $\mathbf{(T_l)}$. It follows from Theorem \ref{conditions T} that there is an element $\lambda \in \mathbb{C}$ and a linear map $\mu :B(\mathcal{H}) \rightarrow \mathbb{C}$ such that $\phi(A)=\lambda A +\mu (A)I$ for all $A\in B(\mathcal{H})$ and $\mu (p_n(A_1, A_2,\cdots , A_n ))=0$ for all $A_1, A_2,\cdots , A_n \in B(\mathcal{H})$ with $A_1 A_2=0$. By Lemma \ref{central map} we have \[\mu(p_n(A_1,E,A_3, \cdots A_n))=0\]
for all idempotents $E\in B(\mathcal{H})$ and all $A_1, A_3, \cdots ,A_n\in B(\mathcal{H})$. If $\mathcal{H}$ is finite-dimensional, then $B(\mathcal{H})= M_n(\mathbb{C})$; the full matrix algebra over $\mathbb{C}$. Since each element of $M_n(\mathbb{C})$ is equal to a linear combination of its own idempotent elements, it follows that $\mu (p_n(A_1, A_2,\cdots , A_n ))=0$ for all $A_1, A_2,\cdots , A_n \in B(\mathcal{H})= M_n(\mathbb{C})$. Let $\mathcal{H}$ be infinite dimensional. By \cite{pe} every operator $A\in B(\mathcal{H})$ is a sum of five idempotents. Now we get $\mu (p_n(A_1, A_2,\cdots , A_n ))=0$ for all $A_1, A_2,\cdots , A_n \in B(\mathcal{H})$. So in any case we conclude that $\phi(A)=\lambda A +\mu (A)I$ for all $A\in B(\mathcal{H})$ and $\mu (p_n(A_1, A_2,\cdots , A_n ))=0$ for all $A_1, A_2,\cdots , A_n \in B(\mathcal{H})$.
\end{proof}
For arbitrary unital standard operator algebras on Banach spaces, we have to assume the linear mappings which are  continuous under the strong operator topology. 
\begin{prop}\label{sot}
Let $\U$ be a unital standard operator algebra over a complex Banach space $\mathcal{X}$ with $dim \mathcal{X}\geq2$. Suppose that $\phi :\U \rightarrow \U$ is a continuous linear mapping under the strong operator topology. The following are equivalent for $n\geq 2$:
\begin{itemize}
\item[(i)] $\phi$ is a Lie $n$-centralizer
\item[(ii)] $\phi$ satisfies $\mathbf{(N)}$;
\item[(iii)] $\phi$ satisfies $\mathbf{(N_l)}$;
\item[(iv)] $\phi$ satisfies $\mathbf{(N_r)}$;
\item[(v)] $\phi$ satisfies $\mathbf{(T)}$;
\item[(vi)] $\phi$ satisfies $\mathbf{(T_l)}$;
\item[(vii)] $\phi$ satisfies $\mathbf{(T_r)}$;
\item[(viii)] There exist an element $\lambda \in \mathbb{C}$ and a linear map $\mu :\U\rightarrow \mathbb{C}$ such that $\phi(A)=\lambda A +\mu (A)I$ for all $A\in \U$ and $\mu (p_n(A_1, A_2,\cdots , A_n ))=0$ for all $A_1, A_2,\cdots , A_n \in \U$.
\end{itemize}
\end{prop}
\begin{proof}
Let $\phi$ satisfies $\mathbf{(T_l)}$. By Theorem \ref{conditions T} there is an element $\lambda \in \mathbb{C}$ and a linear map $\mu :\U\rightarrow \mathbb{C}$ such that $\phi(A)=\lambda A +\mu (A)I$ for all $A\in \U$ and $\mu (p_n(A_1, A_2,\cdots , A_n ))=0$ for all $A_1, A_2,\cdots , A_n \in \U$ with $A_1 A_2=0$. From Lemma \ref{central map} it follows that $\mu(p_n(A_1,E,A_3, \cdots A_n))=0$ for all idempotents $E\in \U$ and all $A_1, A_2, \cdots ,A_n\in \U$. Since every finite-rank operator on $\mathcal{X}$ is a linear combination of rank-one idempotents, we arrive at 
\[\mu(p_n(A_1,T,A_3, \cdots A_n))=0\]
for all $A_1, A_3,\cdots , A_n \in \U$ and $T\in F(\mathcal{X})$. Let $A_2 \in \U$. From the fact that $\overline{F(\mathcal{X})}^{SOT}=B(\mathcal{X})$, there is a net $(T_{\alpha})_{\alpha\in\Gamma}$ in $F(\mathcal{X})$ such that $T_{\alpha}\xrightarrow{SOT} A_2$. By the fact that the product in $B(\mathcal{X})$ is separately continuous with respect to the strong operator topology and $\phi$ is a continuous linear mapping under the strong operator topology, we get $\mu (p_n(A_1, A_2,\cdots , A_n ))=0$ for all $A_1, A_2,\cdots , A_n \in \U$. So we conclude that $\phi(A)=\lambda A +\mu (A)I$ for all $A\in \U$ and $\mu (p_n(A_1, A_2,\cdots , A_n ))=0$ for all $A_1, A_2,\cdots , A_n \in \U$. In this way, we obtain proof (vi)$\Rightarrow$(viii). Considering the same points mentioned at the beginning of the proof of Proposition \ref{hilbert}, the proof is completed.
\end{proof}
According to Example \ref{n-commuting not other}, an $n$-commuting linear map on a unital algebra is not necessarily a Lie $n$-centralizer. In the next result, we see that an $n$-commuting linear map is a Lie $n$-centralizer on non-unital standard operator algebras over infinite dimensional complex Banach spaces.
\begin{thm}\label{n-lie non-unital}
Let $\U$ be a standard operator algebra over a infinite dimensional complex Banach space $\mathcal{X}$ without the identity operator $I$. Assume that $\phi :\U \rightarrow \U$ is a linear map. The following are equivalent for $n\geq 2$:
\begin{itemize}
\item[(i)] $\phi$ is a Lie $n$-centralizer;
\item[(ii)] $\phi$ is an $n$-commuting map;
\item[(iii)] There exist an element $\lambda \in \mathbb{C}$ such that $\phi(A)=\lambda A $ for all $A\in \U$.
\end{itemize}
\end{thm}
\begin{proof}
 (i)$\Rightarrow$(ii) and (iii)$\Rightarrow$(i) are clear.
 \par 
 (ii)$\Rightarrow$(iii): Let $\U^{\sharp}$ be the unitization of $\U$. Hence $\U^{\sharp}$ is a unital standard operator algebra. We extend $\phi$ to the linear mapping $\hat{\phi}:\U^{\sharp}\rightarrow \U^{\sharp}$ with the following definition
 \[\hat{\phi}(A+\gamma I):= \phi(A) \quad (A+\gamma I\in \U^{\sharp}).\]
 By the fact that $\phi$ is an $n$-commuting map, for any $A+\gamma I\in \U^{\sharp}$ and $n\geq 2$ we have
 \begin{equation*}
\begin{split}
p_{n}(\hat{\phi}(A+\gamma I),A+\gamma I,\cdots , A+\gamma I  ) & =p_{n-1}([\hat{\phi}(A+\gamma I),A+\gamma I],\cdots ,  A+\gamma I  )\\&
= p_{n-1}([\phi(A),A+\gamma I],\cdots ,  A+\gamma I  )\\&
= p_{n-1}([\phi(A),A],A+\gamma I ,\cdots ,  A+\gamma I )\\&
= p_{n-1}([\phi(A),A],A ,\cdots ,  A+\gamma I  )\\&
\, \, \,\vdots \\&
= p_{n-1}([\phi(A),A],A ,\cdots ,  A ) )\\&
=p_{n}(\phi(A),A,\cdots ,  A  )\\&
=0.
\end{split}
\end{equation*}
So $\hat{\phi}$ is an $n$-commuting map on $\U^{\sharp}$, and it follows from Lemma \ref{n-comuting standard} that there is an element $\lambda \in \mathbb{C}$ and a linear map $\mu :\U^{\sharp} \rightarrow \mathbb{C}$ such that $\hat{\phi}(A+\gamma I)=\lambda (A+\gamma I) +\mu (A+\gamma I)I$ for all $A+\gamma I\in \U^{\sharp}$. Therefore $\phi(A)=\lambda A+(\lambda\gamma  +\mu (A+\gamma I))I$. Since $\phi(A)\in \U$, it follows that $(\lambda\gamma  +\mu (A+\gamma I))I\in \U$. Considering that $\U $ is non-unital, we see that $\lambda\gamma  +\mu (A+\gamma I)=0$, and hence $\phi(A)=\lambda A$ for all $A\in \U$. The proof is complete.
\end{proof}
We mention the above theorem for $n=2$ separately in the following corollary.
\begin{cor}
Let $\U$ be a standard operator algebra over a infinite dimensional complex Banach space $\mathcal{X}$ without the identity operator $I$. Assume that $\phi :\U \rightarrow \U$ is a linear map. The following are equivalent:
\begin{itemize}
\item[(i)] $\phi$ is a Lie centralizer;
\item[(ii)] $\phi$ is a commuting map;
\item[(iii)] There exist an element $\lambda \in \mathbb{C}$ such that $\phi(A)=\lambda A $ for all $A\in \U$.
\end{itemize}
\end{cor}
\subsection*{Declarations}
\begin{itemize}
\item[•] All authors contributed to the study conception and design and approved the final manuscript. 
\item[•] The authors declare that no funds, grants, or other support were received during the preparation of this manuscript.
\item[•] The authors have no relevant financial or non-financial interests to disclose.
\item[•] Data sharing not applicable to this article as no datasets were generated or analysed during the current study. 
\end{itemize}

\subsection*{Acknowledgment}
The author thanks the referee(s) for careful reading of the manuscript and for helpful suggestions.
\bibliographystyle{amsplain}
\bibliography{xbib}

\end{document}